\documentclass[10pt,a4paper]{amsart}
\usepackage[english]{babel}

\usepackage{amsmath,amssymb,amsthm}

 \theoremstyle{plain}
\newtheorem{theorem}{Theorem}[section]
\newtheorem{prop}[theorem]{Proposition}
\newtheorem{corollary}[theorem]{Corollary}
\newtheorem{lemma}[theorem]{Lemma}

\theoremstyle{definition}
\newtheorem{remark}[theorem]{Remark}

\newtheorem{example}[theorem]{Example}
\newtheorem{definition}[theorem]{Definition}

\newcommand{\R}{\mathbb{R}}
\newcommand{\N}{\mathbb{N}}
\newcommand{\eps}{\varepsilon}
\newcommand{\ext}[1][X^*]{\ensuremath{\mathrm{ext}(B_{#1})}}
\newcommand{\extr}{\ensuremath{\mathrm{ext}}}

 \renewcommand{\leq}{\leqslant}
 \renewcommand{\geq}{\geqslant}

 \DeclareMathOperator{\iso}{Iso}
 
 \DeclareMathOperator{\Id}{\mathrm{Id}}
 
 \DeclareMathOperator{\supp}{\mathrm{Supp}}

\begin{document}

\title[Isometries on extremely non-complex Banach spaces]
{Isometries on Extremely non-complex Banach spaces}

\author{Piotr Koszmider}

\address{Instytut Matematyki Politechniki \L\'odzkiej,
ul.\ W\'olcza\'nska 215, 90-924 \L\'od\'z, Poland}
\email{\texttt{pkoszmider.politechnika@gmail.com}}

\author{Miguel Mart\'{\i}n}
\author{Javier Mer\'{\i}}
 \address{Departamento de An\'{a}lisis Matem\'{a}tico \\ Facultad de
 Ciencias \\ Universidad de Granada \\ 18071 Granada, Spain}
 \email{\texttt{mmartins@ugr.es, jmeri@ugr.es}}
\thanks{The first author was partially supported by Polish Ministry of
Science and Higher Education research grant \mbox{N N201
386234}. The second and the third authors were partially
supported by Spanish MEC project no.\
 MTM2006-04837 and Junta de Andaluc\'{\i}a grants FQM-185 and P06-FQM-01438.}

\date{April 6th, 2009. Revised: January 19th, 2010.}

\dedicatory{}

\thispagestyle{empty}

\begin{abstract}
Given a separable Banach space $E$, we construct an extremely
non-complex Banach space (i.e.\ a space satisfying that $\|\Id +
T^2\|=1+\|T^2\|$ for every bounded linear operator $T$ on it)
whose dual contains $E^*$ as an $L$-summand. We also study
surjective isometries on extremely non-complex Banach spaces and
construct an example of a real Banach space whose group of
surjective isometries reduces to $\pm\Id$, but the group of
surjective isometries of its dual contains the group of isometries
of a separable infinite-dimensional Hilbert space as a subgroup.
\end{abstract}

\maketitle

\section{Introduction}
All the Banach spaces in this paper will be real. Given a
Banach space $X$, we write $X^*$ for the topological dual,
$L(X)$ for the space of all bounded linear operators, $W(X)$
for the space of weakly compact operators and $\iso(X)$ for the
group of surjective isometries.

A Banach space $X$ is said to be \emph{extremely non-complex}
if the norm equality
\begin{equation*}
\|\Id + T^2\|=1 + \|T^2\|
\end{equation*}
holds for every $T\in L(X)$. This concept was introduced very
recently by the authors in \cite{KMM}, where several different
examples of $C(K)$ spaces are shown to be extremely
non-complex. For instance, this is the case for some perfect
compact spaces $K$ constructed by the first author
\cite{Koszmider} such that $C(K)$ has few operators (in the
sense that every operator is a weak multiplier). There are
other examples of extremely non-complex $C(K)$ spaces which
contain complemented copies of $\ell_\infty$ or $C(2^\omega)$
(and so, they do not have few operators). It is trivial that
$X=\R$ is extremely non-complex. The existence of
infinite-dimensional extremely non-complex Banach spaces had
been asked in \cite[Question~4.11]{KaMaMe}, where possible
generalizations of the Daugavet equation were investigated. We
recall that an operator $S$ defined on a Banach space $X$
satisfies the \emph{Daugavet equation} \cite{Dau} if
\begin{equation*}
\|\Id + S\|= 1 + \|S\|
\end{equation*}
and that the space $X$ has the \emph{Daugavet property}
\cite{KSSW} if the Daugavet equation holds for every rank-one
operator on $X$. We refer the reader to
\cite{AbrAli-1,AbrAli-2,KSSW,WerSur} for background on the
Daugavet property. Let us observe that $X$ is extremely
non-complex if the Daugavet equation holds for the square of
every (bounded linear) operator on $X$. Spaces $X$ in which the
square of every rank-one operator on $X$ satisfies the Daugavet
equation are studied in \cite{Oik2007}, where it is shown that
their unit balls do not have strongly exposed points. In
particular, the unit ball of an extremely non-complex Banach
space (of dimension greater than one) does not have strongly
exposed points and, therefore, the space does not have the
Radon-Nikod\'{y}m property, even the more it is not reflexive.

The name of extremely non-complex comes from the fact that a real
Banach space $X$ is said to have a \emph{complex structure} if
there exists $T\in L(X)$ such that $T^2=-\Id$, so extremely
non-complex spaces lack of complex structures in a very strong
way. Let us also comment that no hyperplane (actually no
finite-codimensional subspace) of an extremely non-complex Banach
space admits a complex structure. The existence of
infinite-dimensional real Banach spaces admitting no complex
structure is known since the 1950's, when J.~Dieudonn\'{e}
\cite{Dieu1952} showed that this is the case of the James' space
$\mathcal{J}$. We refer the reader to the very recent papers by
V.~Ferenczi and E.~Medina~Galego \cite{Ferenczi-Adv,FerenMedina}
and references therein for a discussion about complex structures
on spaces and on their hyperplanes.

Our first goal in this paper is to present examples of extremely
non-complex Banach spaces which are not isomorphic to $C(K)$ spaces.
Namely, it is proved in section~\ref{sec:CE-when-CK-hasfewoperators}
that if $K$ is a compact space such that all operators on $C(K)$ are
weak multipliers, $L$ is a closed nowhere dense subset of $K$ and
$E$ is a subspace of $C(L)$, then the space
$$
C_E(K\|L)=\{f\in C(K)\,:\, f|_L\in E\}
$$
is extremely non-complex (see section~\ref{sec:notation-CEKL}
for the definitions and basic facts about this kind of spaces).
It is also shown that there are extremely non-complex
$C_E(K\|L)$ spaces which are not isomorphic to $C(K')$ spaces.
On the other hand, some  spaces $C_E(K\|L)$ are isometric to
spaces $C(K')$ for some compact $K'$. This can be used to note
that the extremely non-complex spaces of the form $C_E(K\|L)$
may have many operators besides weak multipliers (see
Remark~\ref{view-as-cofk}).

The next aim is to show that $\iso(X)$ is a ``discrete'' Boolean
group when $X$ is extremely non-complex. Namely, we show that
$T^2=\Id$ for every $T\in \iso(X)$ (i.e.\ $\iso(X)$ is a Bolean
group and so, it is commutative) and that $\|T_1-T_2\|\in \{0,2\}$
for every $T_1,T_2\in \iso(X)$. Next, we discuss the relationship
with the set of all unconditional projections on $X$ and the
possibility of this set to be a Boolean algebra. This is the content
of section~\ref{sec:isometries-on-X}. In
section~\ref{sec:isomtries-CEKL} we particularize these results to
spaces $C_E(K\|L)$ which are extremely non-complex, getting
necessary conditions on the elements of $\iso\bigl(C_E(K\|L)\bigr)$.
In particular, if $C(K)$ is extremely non-complex, we show that the
only homeomorphism from $K$ to $K$ is the identity, obtaining that
$\iso\bigl(C(K)\bigr)$ is isomorphic to the Boolean algebra of
clopen sets of $K$.

Section~\ref{sec:main-example} is devoted to apply all the results
above to get an example showing that the behavior of the group of
isometries with respect to duality can be extremely bad. Namely,
we show that for every separable Banach space $E$, there is a
Banach space $\widetilde{X}(E)$ such that
$\iso\bigl(\widetilde{X}(E)\bigr)=\{\Id,-\Id\}$ and
$\widetilde{X}(E)^*=E^*\oplus_1 Z$, so
$\iso\bigl(\widetilde{X}(E)^*\bigr)$ contains $\iso(E^*)$ as a
subgroup. To do so, we have to modify a construction of a
connected compact space $K$ with few operators given by the first
named author in \cite[\S5]{Koszmider} and use our construction of
$C_E(K\|L)$ for a nowhere dense $L\subset K$. For the special case
$E=\ell_2$, we have
$\iso\bigl(\widetilde{X}(\ell_2)\bigr)=\{\Id,-\Id\}$, while
$\iso\bigl(\widetilde{X}(\ell_2)^*\bigr)$ contains infinitely many
uniformly continuous one-parameter semigroups of surjective
isometries. Let us comment that, in sharp contrast with the
examples above, when a Banach space $X$ is \emph{strongly unique
predual}, the group $\iso(X^*)$ consists exactly of the conjugate
operators to the elements of $\iso(X)$. Quite a lot of spaces are
actually strong unique preduals. We refer the reader to
\cite{Godefroy} for more information.

We finish this introduction with some needed notation. If $X$ is a
Banach space, we write $B_X$ to denote the closed unit ball of $X$
and given a convex subset $A\subseteq X$, $\extr(A)$ denotes the
set of extreme points of $A$. A closed subspace $Z$ of $X$ is an
\emph{$L$-summand} if $X=Z\oplus_1 W$ for some closed subspace $W$
of $X$, where $\oplus_1$ denotes the $\ell_1$-sum. A closed
subspace $Y$ of a Banach space $X$ is said to be an
\emph{$M$-ideal} of $X$ if the annihilator $Y^\perp$ of $Y$ is an
$L$-summand of $X^*$. We refer the reader to \cite{H-W-W} for
background on $L$-summands and $M$-ideals.

\section{Notation and preliminary results on the spaces $C_E(K\|L)$}
\label{sec:notation-CEKL} All along the paper, $K$ will be a
(Hausdorff) compact (topological) space and $L\subseteq K$ will
stand for a nowhere dense closed subset. Given a closed
subspace $E$ of $C(L)$, we will consider the subspace of $C(K)$
given by
$$
C_E(K\|L)=\{f\in C(K)\,:\, f|_L\in E\}.
$$
This notation is compatible with the Semadeni's book
\cite[II.~4]{Sem} notation of
$$
C_0(K\|L)=\{f\in C(K)\,:\, f|_L=0\}.
$$
This space can be identified with the space $C_0(K\setminus L)$
of those continuous functions $f:K\setminus L\longrightarrow
\R$ vanishing at infinity.

By the Riesz representation theorem, the dual space of $C(K)$
is isometric to the space $M(K)$ of Radon measures on $K$,
i.e.\ signed, Borel, scalar-valued, countably additive and
regular measures. More precisely, given $\mu\in M(K)$ and $f\in
C(K)$, the duality is given by
$$
\mu(f)=\int f d\mu.
$$
We recall that $C_0(K\|L)$ is an $M$-ideal of $C(K)$
\cite[Example~I.1.4(a)]{H-W-W}, meaning that $C_0(K\|L)^\perp$
is an $L$-summand in $C(K)^*$. This fact allows to show the
following well-known result.

\begin{lemma}\label{lemma-dual-of-czero}
$C_0(K\|L)^*\equiv\{\mu\in M(K)\,:\, |\mu|(L)=0\}$.
\end{lemma}

\begin{proof}
Since $C_0(K\|L)$ is an $M$-ideal in $C(K)$, Proposition~1.12 and
Remark~1.13 of \cite{H-W-W} allow us to identify $C_0(K\|L)^*$ with
the subspace of $C(K)^*=M(K)$ given by
\begin{equation*}
C_0(K\|L)^{\#}=\{\mu\in M(K)\,:\, |\mu|(K)=|\mu|(K\setminus L)\}=
\{\mu\in M(K)\,:\, |\mu|(L)=0\}.\qedhere
\end{equation*}
\end{proof}

When we consider the space $C_E(K\|L)$, it still makes sense to
talk about functionals corresponding to the measures on $K$,
namely one understands them as the restriction of the
functional from $C(K)$ to $C_E(K\|L)$. However, given a
functional belonging to $C_E(K\|L)^*$ one may have several
measures on $K$ associated with it. The next result describes
the dual of a $C_E(K\|L)$ space for an arbitrary $E\subseteq
C(L)$. It is worth mentioning that its proof is an extension of
that appearing in \cite[Theorem~3.3]{MarIso}.

\begin{lemma}\label{lemma-decomposition-dual}
Let $K$ be a compact space, $L$ a closed subset of $K$ and
$E\subseteq C(L)$. Then,
$$
C_E(K\|L)^*\equiv C_0(K\|L)^*\oplus_1 C_0(K\|L)^\perp  \equiv
C_0(K\|L)^*\oplus_1 E^* .
$$
\end{lemma}

\begin{proof}
We write $P: C(K)\longrightarrow C(L)$ for the restriction
operator, i.e.
$$
[P(f)](t)=f(t) \qquad (t\in L,\ f\in C(K)).
$$
Then, $C_0(K\|L)=\ker P$ and $C_E(K\|L)=\{f\in C(K) \ : \
P(f)\in E\}$. Since $C_0(K\|L)$ is an $M$-ideal in $C(K)$, it
is a fortiori an $M$-ideal in $C_E(K\|L)$ by
\cite[Proposition~I.1.17]{H-W-W}, meaning that
$$
C_E(K\|L)^*\equiv C_0(K\|L)^*\oplus_1 C_0(K\|L)^{\perp}\equiv
C_0(K\|L)^*\oplus_1 \bigl[C_E(K\|L)/C_0(K\|L)\bigr]^*.
$$
Now, it suffices to prove that the quotient
$C_E(K\|L)/C_0(K\|L)$ is isometrically isomorphic to $E$. To do
so, we define the operator $\Phi: C_E(K\|L)\longrightarrow E$
given by $\Phi(f)=P(f)$ for every $f\in C_E(K\|L)$. Then $\Phi$
is well defined, $\|\Phi\|\leq 1$, and $\ker \Phi=C_0(K\|L)$.
To see that the canonical quotient operator
$\widetilde{\Phi}:C_E(K\|L)/C_0(K\|L)\longrightarrow E$ is a
surjective isometry, it suffices to show that
$$
\Phi\bigl(\{f\in C_E(K\|L) \, : \, \|f\|<1 \}\bigr)=\{g\in E \, : \, \|g\|<1\}.
$$
Indeed, the left-hand side is contained in the right-hand side
since $\|\Phi\|\leq1$. Conversely, for every $g\in E\subseteq
C(L)$ with $\|g\|<1$, we just use Tietze's extension theorem to
find $f\in C(K)$ such that $\Phi(f)=f|_{L}=g$ and
$\|f\|=\|g\|$.
\end{proof}

If $\phi\in C_E(K\|L)^*$, by the above lemma we have
$\phi=\phi_1+\phi_E$ with $\phi_1\in C_0(K\|L)^*$ and $\phi_E\in
C_0(K\|L)^\perp\equiv E^*$. Observe that $\phi_1$ can be
isometrically associated with a measure on $K\setminus L$ by
Lemma~\ref{lemma-dual-of-czero}. which will be denoted
$\phi|_{K\setminus L}$. Given a subset $A\subseteq K$ satisfying
$A\cap L=\emptyset$, $\phi|_A$ will stand for the measure
$\phi|_{K\setminus L}$ restricted to $A$.

The next result is a straightforward consequence of the above
two lemmas.

\begin{lemma}\label{lemma-norm-in-the-dual}
Let $\phi\in C_E(K\|L)^*$ and $x\in K\setminus L$. Then
$$
\|\phi\|=\left\|\phi|_{\{x\}}\right\|+\left\|\phi|_{K\setminus
(L\cup\{x\})}\right\| + \|\phi_E\|.
$$
\end{lemma}

The next easy lemma describes the set of extreme points in the unit
ball of the dual of $C_E(K\|L)$ and gives a norming set for
$C_E(K\|L)$. We recall that a subset $A$ of the unit ball of the
dual of a Banach space $X$ is said to be \emph{norming} if
$$
\|x\|=\sup\{|\phi(x)|\,:\,\phi\in A\} \qquad (x\in X).
$$

\begin{lemma}\label{lemma:1-norming-CEKL}
Let $K$ be a compact space, $L$ a nowhere dense closed subset
and $E\subseteq C(L)$. We consider the set
$$
\mathcal{A}=\bigl\{\theta\,\delta_y\,
: \, y\in K\setminus L,\,\theta\in\{-1,1\}\bigr\}\subset C_E(K\|L)^*.
$$
Then:
\begin{enumerate}
\item[(a)] $\ext[C_E(K\|L)^*]=\mathcal{A}\cup \ext[E^*]$.
\item[(b)] $\mathcal{A}$ is norming for $C_E(K\|L)$.
\item[(c)] Therefore, $\mathcal{A}$ is weak$^*$-dense in
    $\ext[C_E(K\|L)^*]$.
\end{enumerate}
\end{lemma}

\begin{proof}
By Lemma~\ref{lemma-decomposition-dual} and the description of
the extreme points of the unit ball of an $\ell_1$-sum of
Banach spaces \cite[Lemma~I.1.5]{H-W-W}, we have
$$
\ext[C_E(K\|L)^*]=\ext[C_0(K\setminus L)^*] \cup \ext[E^*].
$$
It suffices to recall that $\ext[C_0(K\setminus
L)^*]=\mathcal{A}$ (see \cite[Theorem~2.3.5]{Fle-Jam1} for
instance) to get (a). The fact that $\mathcal{A}$ is norming
for $C_E(K\|L)$ is a direct consequence of the fact that
$K\setminus L$ is dense in $K$. Finally, every norming set is
weak$^*$-dense in $\ext[C_E(K\|L)^*]$ by the Hahn-Banach
theorem and the reversed Krein-Milman theorem.
\end{proof}

We introduce one more ingredient which will play a crucial role
in our arguments. Given an operator $U\in
L\big(C_E(K\|L)^*\big)$, we consider the function
$$
g_U:K\setminus L \longrightarrow [-\|U\|,\|U\|], \qquad
g_U(x)=U(\delta_x)(\{x\})\qquad \bigl(x\in K\setminus L\bigr).
$$
This obviously extends to operators on $C_E(K\|L)$ by passing
to the adjoint, that is, for $T\in L\big(C_E(K\|L)\big)$ one
can consider $g_{T^*}:K\setminus L\longrightarrow
[-\|T\|,\|T\|]$. This is a generalization of a tool  used in
\cite{Wer0} under the name ``stochastic kernel''. One of the
results in that paper can be generalized to the following.

\begin{lemma}\label{lemma-suf-condition}
Let $K$ be a compact space, $L$ a nowhere dense closed subset
of $K$, $E\subseteq C(L)$, and $T\in L\bigl(C_E(K\|L)\bigr)$.
If the set $\{x\in K\setminus L \, : \, g_{T^*}(x)\geq 0\}$ is
dense in $K\setminus L$, then $T$ satisfies the Daugavet
equation.
\end{lemma}

\begin{proof} We use
Lemmas \ref{lemma-decomposition-dual},
\ref{lemma-norm-in-the-dual} and \ref{lemma:1-norming-CEKL} to
get
\begin{align}\label{eq:lemma-suf-condition}
\|\Id+T^*\|&\geq \underset{x\in K\setminus L}{\sup}\|\delta_x+T^*(\delta_x)\|\\
&=\underset{x\in K\setminus
L}{\sup}|1+T^*(\delta_x)(\{x\})|+\|T^*(\delta_x)|_{(K\setminus (L\cup\{x\}))}\|+
\|T^*(\delta_x)|_E\|\notag\\
&=\underset{x\in K\setminus L}{\sup}|1+T^*(\delta_x)(\{x\})|
-|T^*(\delta_x)(\{x\})|+\|T^*(\delta_x)\| \notag\\
&\geq\underset{x\in K\setminus
L}{\sup}1+T^*(\delta_x)(\{x\})-|T^*(\delta_x)(\{x\})|+\|T^*(\delta_x)\|.\notag
\end{align}
Now, we claim that the set $\{x\in K\setminus L\,:\,
\|T^*(\delta_x)\|>\|T\|-\eps\}$ is nonempty and open in
$K\setminus L$ for every $\eps>0$. Indeed, take a norm one
function $f\in C_E(K\|L)$ such that $\|T(f)\|>\|T\|-\eps$ and
use the fact that $K\setminus L$ is dense in $K$ to find $x\in
K\setminus L$ satisfying
$$
\|T^*(\delta_x)\|\geq|T^*(\delta_x)(f)|=|T(f)(x)|>\|T\|-\eps.
$$
To show that $\{x\in K\setminus L\,:\,
\|T^*(\delta_x)\|>\|T\|-\eps\}$ is open we prove that the
mapping
$$
x\longmapsto \|T^*(\delta_x)\| \qquad (x\in K\setminus L)
$$
is lower semicontinuous. To do so, since $T^*$ is weak$^*$
continuous and $\|\cdot\|$ is weak$^*$ lower semicontinuous, it
suffices to observe that the mapping which sends $x$ to $\delta_x$
is continuous with respect to the weak$^*$ topology on
$C_E(K\|L)^*$. But this is so since, for $a\in \R$ and $f\in
C_E(K\|L)$, the preimage of the subbasic set $\{\phi\in
C_E(K\|L)^*\,:\, \phi(f)<a\}$ in this topology is $\{x\in K\setminus
L\,:\, f(x)<a\}$ which is open in $K\setminus L$.

To finish the proof, we use the hypothesis to find $x_0\in
K\setminus L$ satisfying
$$
\|T^*(\delta_{x_0})\|>\|T\|-\eps \qquad \text{and} \qquad
g_{T^*}(x_0)=T^*(\delta_{x_0})(\{x_0\})\geq 0
$$
and we use it in \eqref{eq:lemma-suf-condition} to obtain
$$
\|\Id+T^*\|\geq 1+\|T^*(\delta_{x_0})\|>1+\|T\|-\eps
$$
which implies that $\|\Id+T\|=\|\Id+T^*\|\geq 1+\|T\|$ since
$\eps$ was arbitrary.
\end{proof}

\section{Spaces $C_E(K\|L)$ when $C(K)$ has few
operators}\label{sec:CE-when-CK-hasfewoperators} Let us start
fixing some notation and terminology that will be used throughout
the section. If $g : K\longrightarrow \R$ is a bounded Borel
function, we will consider the operator $g\Id : C(K)^*
\longrightarrow C(K)^*$ which sends the functional which is the
integration of a function $f\in C(K)$ with respect to a measure
$\mu$ to the functional which is the integration of the product
$fg$ with respect to $\mu$.

\begin{definition}\label{def-weak-multiplication-weak-multiplier}
Let $K$ be a compact space and $T\in L(C(K))$. We say that $T$
is a \emph{weak multiplier} if $T^*=g\Id+S$ where
$g:K\longrightarrow \R$ is a bounded Borel function  on $K$ and
$S\in W\big(C(K)^*\big)$.
\end{definition}

This definition was given in \cite{Koszmider} in an equivalent
form (see \cite[Definition~2.1]{Koszmider} and
\cite[Theorem~2.2]{Koszmider}).

\begin{definition} We say that an open set $V\subseteq K$ is
\emph{compatible} with $L$ if and only if $L\subseteq V$ or $L\cap
\overline{V}=\emptyset$. In the first case, the notation
$C_E(\overline{V}\|L)$ has the same meaning as in the previous
section. If $L\cap \overline{V}=\emptyset$, we will write
$C_E(\overline{V}\|L)$ just to denote $C(\overline{V})$. Let us
also observe that if $L\subseteq V$, then
$$
C_E(\overline{V}\|L)^*\equiv
C_0(\overline{V}\|L)^*\oplus_1C_0(\overline{V}\|L)^\perp\equiv
C_0(\overline{V}\|L)^*\oplus_1 E^*
$$
since Lemma~\ref{lemma-decomposition-dual} applies to
$\overline{V}$.
\end{definition}

Given an open set $V\subseteq K$ compatible with $L$, we consider
the restriction operator
$P^{\overline{V}}:C_E(K\|L)^*\longrightarrow
C_E(\overline{V}\|L)^*$ given by
$$
P^{\overline{V}}(\phi)=\phi|_{\overline{V}\setminus L}+\phi_E
$$
for $\phi=\phi|_{K\setminus L}+\phi_E$ where $\phi|_{K\setminus
L}\in C_0(K\|L)^*$ and $\phi_E\in C_0(K\|L)^\perp\equiv E^*$.
Observe that $\phi_E$ can also be viewed as an element of
$C_0(\overline{V}\|L)^\perp$ since the spaces
$C_0(\overline{V}\|L)^\perp$ and $C_0(K\|L)^\perp$ are isometrically
isomorphic (both coincide with $E^*$).

Given an open set $V\subseteq K$ compatible with $L$ and
$h:K\longrightarrow[0,1]$ a continuous function constant on $L$ with
support included in $V$, we denote by $P_{\overline{V}}:
C_E(K\|L)\longrightarrow C_E(\overline{V}\|L)$ and
$I_{h,\overline{V}}:C_E(\overline{V}\|L)\longrightarrow C_E(K\|L)$
the operators defined by
$$
P_{\overline{V}}(f)=f|_{\overline{V}}\qquad \text{and}\qquad
I_{h,\overline{V}}(\widetilde{f})=h\widetilde{f}
$$
for $f\in C_E(K\|L)$ and $\widetilde{f}\in C_E(\overline{V}\|L)$
respectively. We observe that $I_{h,\overline{V}}$ is well defined,
that is, $h\widetilde{f}$ is a function in $C(K)$ with
$h\widetilde{f}|_L \in E$ (indeed, $h\widetilde{f}$ is continuous in
$V$ as a product of two continuous functions and it is continuous in
$K\setminus \supp(h)$ as a constant function, since these two sets
form an open cover of $K$ we have that $h\widetilde{f}$ is
continuous in $K$; being $h$ constant on $L$, it is clear that
$h\widetilde{f}|_L \in E$). Finally, If $V_1\subseteq K$ is an open
set compatible with $L$ satisfying $\overline{V}\subseteq V_1$ we
will also use the notation $P_{\overline{V}}$ for the restriction
operator from $C_E(\overline{V}_1\|L)$ to $C_E(\overline{V}\|L)$. In
the next result we gather some easy facts concerning these
operators.

\begin{lemma}\label{lemma-duals-inlusion-multiplication}
Let $V$ and $h$ be as above. Then, the following hold:
\begin{enumerate}
\item[(a)] $P_{\overline{V}}^*(\phi)(f)=\phi(f|_{\overline{V}})$ for
$\phi\in C_E(\overline{V}\|L)^*$ and $f\in C_E(K\|L)$.
\item[(b)]$I_{h,\overline{V}}^*(\phi)(\widetilde{f})=\phi(\widetilde{f}h)$
for $\phi\in C_E(K\|L)^*$ and $\widetilde{f}\in
C_E(\overline{V}\|L)$.
\item[(c)] If $V_0$ is an open set such that $\overline{V}_0\subseteq
V$ and $h|_{\overline{V}_0}\equiv 1$, then
$(I_{h,\overline{V}}^*\,P_{\overline{V}}^*)(\mu)=\mu$ for every
$\mu\in C(K)^*$ with $\supp(\mu)\subseteq V_0$.
\item[(d)] If $E=C(L)$, then $C_E(K\|L)=C(K)$,
$C_E(\overline{V}\|L)=C(\overline{V})$, and
$P^{\overline{V}}P_{\overline{V}}^*(\mu)=\mu$
    for every $\mu\in C(\overline{V})^*$.
\end{enumerate}
\end{lemma}
\begin{proof}
$\mathrm{(a)}$ and $\mathrm{(b)}$ are obvious from the definitions
of the operators. To prove $\mathrm{(c)}$ we fix $\widetilde{f}\in
C_E(\overline{V}\|L)$, $\mu\in C(K)^*$ with $\supp(\mu)\subseteq
V_0$ and we observe that
\begin{align*}
(I_{h,\overline{V}}^*\,P_{\overline{V}}^*)(\mu)(\widetilde{f})&=P_{\overline{V}}^*(\mu)
\big(I_{h,\overline{V}}(\widetilde{f})\big)=P_{\overline{V}}^*(\mu)(\widetilde{f}h)\\&=
\mu\left((\widetilde{f}h)|_{\overline{V}}\right)=\int(\widetilde{f}h)|_{\overline{V}}\,d\mu=
\int_{V_0}\widetilde{f}\,d\mu=\mu(\widetilde{f}).
\end{align*}

The first two assertions of $\mathrm{(d)}$ are obvious. For the
third one, given $\mu\in C(\overline{V})^*$ and $\widetilde{f}\in
C(\overline{V})$, use the regularity of the measure
$P_{\overline{V}}^*(\mu)$ to find an open set $V_n\subseteq K$
satisfying
\begin{equation}\label{eq:lemma-properties-operators}
\overline{V}\subseteq V_n \qquad \text{and} \qquad
\left|P_{\overline{V}}^*(\mu)\right|(V_n\setminus\overline{V})<\frac1n
\end{equation}
for every $n\in\N$. Next, take $f_n\in C(K)$ satisfying
$$
f_n|_{\overline{V}}\equiv \widetilde{f}, \qquad f_n|_{K\setminus
V_n}\equiv 0, \qquad \text{and} \qquad \|f_n\|=\|\widetilde{f}\|
$$
for every $n\in\N$, and observe that
\begin{align*}
P^{\overline{V}}P_{\overline{V}}^*(\mu)(\widetilde{f})&=
\bigl(P_{\overline{V}}^*(\mu)\bigr)|_{\overline{V}}(\widetilde{f})=\int_{\overline{V}}
\widetilde{f}\,d\bigl(P_{\overline{V}}^*(\mu)\bigr)|_{\overline{V}}
\\&=\int_{K}f_n\,dP_{\overline{V}}^*(\mu)-\int_{V_n\setminus\overline{V}}
f_n\,dP_{\overline{V}}^*(\mu)\\&=P_{\overline{V}}^*(\mu)(f_n)-\int_{V_n\setminus\overline{V}}
f_n\,dP_{\overline{V}}^*(\mu)=\mu(\widetilde{f})-\int_{V_n\setminus\overline{V}}
f_n\,dP_{\overline{V}}^*(\mu).
\end{align*}
Therefore, using \eqref{eq:lemma-properties-operators} and letting
$n\rightarrow \infty$, it follows that
$P^{\overline{V}}P_{\overline{V}}^*(\mu)(\widetilde{f})=\mu(\widetilde{f})$.
\end{proof}

Our first application uses the above operators in the simple case in
which $E=C(L)$.

\begin{prop}\label{operators-on-open-of-assymetric}
Let $K$ be a compact space, let $V_0$, $V_1$, and $V_2$ be open
nonempty subsets of $K$ such that $\overline{V}_0\subseteq
V_1$, and let $R: C(\overline{V}_2) \longrightarrow
C(\overline{V}_1)$ be a linear operator. Suppose that all
operators on $C(K)$ are weak multipliers. Then, there are a
Borel function $g:\overline{V}_1 \longrightarrow \R$ with
support included in $\overline{V}_1\cap \overline{V}_2$ and a
weakly compact operator $S: C(\overline{V}_1)^*\longrightarrow
C(\overline{V}_2)^*$ such that
$$
R^*(\mu)=g\mu+S(\mu)
$$
for every $\mu\in C(K)^*$ with $\supp(\mu)\subseteq V_0$.
\end{prop}

\begin{proof}
We fix a continuous function $h: K\longrightarrow [0,1]$ satisfying
$h|_{\overline{V}_0}\equiv 1$ and $h|_{(K\setminus{V}_1)}\equiv 0$
and we define $R_0\in L(C(K))$ by
\begin{equation}\label{eq:prop-assymetric-R0}
R_0(f)=I_{h,\overline{V_1}}\, R P_{\overline{V}_2}(f)
\end{equation}
for $f\in C(K)$. Hence, there are a bounded Borel function
$\widehat{g}:K\longrightarrow \R$ and a weakly compact operator
$\widehat{S}\in L(C(K)^*)$ such that
$R_0^*(\mu)=\widehat{g}\mu+\widehat{S}(\mu)$ for $\mu\in C(K)^*$,
which allows us to write
\begin{equation}\label{eq:prop-assymetric-PV2R0PV1}
P^{\overline{V}_2}R_0^* P^*_{\overline{V}_1}= P^{\overline{V}_2}
\widehat{g}\Id_{C(K)^*} P^*_{\overline{V}_1}+ P^{\overline{V}_2}
\widehat{S} P^*_{\overline{V}_1}.
\end{equation}
We claim that, considering the weakly compact operator given by
$S=P^{\overline{V}_2} \widehat{S} P^*_{\overline{V}_1}$ and defining
the functions $\breve{g}:K \longrightarrow \R$ and $g:\overline{V}_1
\longrightarrow \R$ by
$$
\breve{g}(x)=\begin{cases}\widehat{g}(x) &\text{ if } x\in
\overline{V}_1\cap \overline{V}_2 \\ 0 &\text{ if } x\notin
\overline{V}_1\cap \overline{V}_2
\end{cases} \qquad \text{and} \qquad g=\breve{g}|_{\overline{V}_1}
$$
the following holds for $\mu\in C(\overline{V}_1)^*$\,:
\begin{equation}\label{eq:prop-assymetric-3}
\big(P^{\overline{V}_2}R_0^*
P^*_{\overline{V}_1}\big)(\mu)=g\mu+S(\mu).
\end{equation}
Indeed, for $\mu\in C(\overline{V}_1)^*$ and $f\in
C(\overline{V}_2)$ we observe that
\begin{align*}
\big(P^{\overline{V}_2} \widehat{g}\Id_{C(K)^*}
P^*_{\overline{V}_1}\big)(\mu)(f)&=(\widehat{g}P^*_{\overline{V}_1}
(\mu))|_{\overline{V}_2}(f)=\int_{\overline{V}_2}\widehat{g}f\,dP^*_{\overline{V}_1}(\mu)\\
&=\int_{\overline{V}_1\cap\overline{V}_2}\widehat{g}f\,dP^*_{\overline{V}_1}(\mu)
=\int_{\overline{V}_1\cap\overline{V}_2}\breve{g}f\,dP^*_{\overline{V}_1}(\mu)
=\int_{K}\breve{g}f\,dP^*_{\overline{V}_1}(\mu)
\end{align*}
and, for $n\in\N$, we use Lusin's Theorem (see
\cite[Theorem~21.4]{Munroe}, for instance) to find a compact set
$K_n\subseteq K$ such that
\begin{equation}\label{eq:prop-assymetric-Kn}
\breve{g}|_{K_n} \text{ is continuous on } K_n, \quad |P_{\overline
{V}_1}^*(\mu)|(K\setminus K_n)<\frac1n, \quad \text{and} \quad
|\mu|\bigl(\overline{V}_1\setminus(\overline{V}_1\cap
K_n)\bigr)<\frac1n\,.
\end{equation}
Using Tietze's extension Theorem we may find a continuous function
$g_n:K\longrightarrow \R$ satisfying
$$
g_n|_{K_n}=\breve{g}|_{K_n} \qquad \text{and} \qquad
\|g_n\|=\|\breve{g}|_{K_n}\|\leq\|\breve{g}\|
$$
for every $n\in\N$. Now it is easy to check that
\begin{align*}
\int_{K}\breve{g}f\,dP^*_{\overline{V}_1}(\mu)&=\int_{K}g_nf\,dP^*_{\overline{V}_1}(\mu)
+\int_{K\setminus K_n}(\breve{g}-g_n)f\,dP^*_{\overline{V}_1}(\mu)\\
&=P^*_{\overline{V}_1}(\mu)(g_nf)+\int_{K\setminus
K_n}(\breve{g}-g_n)f\,dP^*_{\overline{V}_1}(\mu)\\
&=\mu\bigl((g_nf)|_{\overline{V}_1}\bigr)+\int_{K\setminus
K_n}(\breve{g}-g_n)f\,dP^*_{\overline{V}_1}(\mu)\\
&=\int_{\overline{V}_1}g_nf d\mu+\int_{K\setminus
K_n}(\breve{g}-g_n)f\,dP^*_{\overline{V}_1}(\mu)\\
&=\int_{\overline{V}_1}gf
d\mu+\int_{\overline{V}_1\setminus(\overline{V}_1\cap
K_n)}(g_n-\breve{g})f d\mu+\int_{K\setminus
K_n}(\breve{g}-g_n)f\,dP^*_{\overline{V}_1}(\mu)
\\
&=\mu(gf)+\int_{\overline{V}_1\setminus(\overline{V}_1\cap
K_n)}(g_n-\breve{g})f d\mu+\int_{K\setminus
K_n}(\breve{g}-g_n)f\,dP^*_{\overline{V}_1}(\mu)
\end{align*}
which, letting $n\rightarrow\infty$ and using
\eqref{eq:prop-assymetric-Kn}, implies that
$$
\int_{K}\breve{g}f\,dP^*_{\overline{V}_1}(\mu)=\mu(gf)
$$
and, therefore,
$$
\big(P^{\overline{V}_2} \widehat{g}\Id_{C(K)^*}
P^*_{\overline{V}_1}\big)(\mu)(f)=\mu(gf).
$$
This, together with \eqref{eq:prop-assymetric-PV2R0PV1} and the
definition of $S$, finishes the proof of the claim. On the other
hand by \eqref{eq:prop-assymetric-R0}, we can write
$$
P^{\overline{V}_2}R^*_0 P^*_{\overline{V}_1}= P^{\overline{V}_2}
P^*_{\overline{V}_2} R^* I_{h,\overline{V}_1}^*
P^*_{\overline{V}_1}.
$$
So, if the support of $\mu$ is included in $V_0$, by
Lemma~\ref{lemma-duals-inlusion-multiplication}, parts (c) and (d)
we obtain
$$
P^{\overline{V}_2} R^*_0 P^*_{\overline{V}_1}(\mu)=R^*(\mu)
$$
and, consequently, $R^*(\mu)=g\mu+S(\mu)$ follows from
\eqref{eq:prop-assymetric-3}.
\end{proof}

\begin{remark}
The result above shows that if every operator on $C(K)$ is a weak
multiplier then, in the above sense, there are also few operators on
$C(\overline{V})$ for $V$ open (since for such a closed set it is
possible to define an appropriate function $h$ as in the proof). In
general, one cannot replace closures of open sets by general closed
sets: it is shown in \cite{Fa} that under CH, there are compact
$K$'s as above which contain $\beta \N$ (and, of course, there are
many operators on $C(\beta \N)\equiv \ell_\infty$). On the other
hand, using the set-theoretic principle $\diamondsuit$, it is also
shown in \cite{Fa} that there are $K$'s such that for every infinite
closed $K'\subseteq K$, all operators on the space $C(K')$ are weak
multipliers.
\end{remark}

\begin{corollary}\label{cor-operators-between-disjoint-sets}
Let $K$ be a compact space, let $V_0$, $V_1$, and $V_2$ be open
nonempty subsets of $K$ such that $\overline{V}_0\subseteq V_1$ and
$\overline{V}_1\cap\overline{V}_2=\emptyset$, and let $R:
C(\overline{V}_2) \longrightarrow C(\overline{V}_1)$ be a linear
operator. Suppose that all operators on $C(K)$ are weak multipliers.
Then, $P_{\overline{V}_0}R$ is weakly compact.
\end{corollary}

\begin{proof} By Proposition~\ref{operators-on-open-of-assymetric}
there is a weakly compact operator
$S:C^*(\overline{V}_1)\longrightarrow C^*(\overline{V}_2)$ such that
$R^*(\mu)=S(\mu)$ for every measure $\mu$ with support included in
$\overline{V}_0$. In other words, $(P_{\overline{V}_0}R)^*=R^*
P_{\overline{V}_0}^*$ is weakly compact and so, by Gantmacher
theorem, $P_{\overline{V}_0}R$ is weakly compact.
\end{proof}

The following result is an easy consequence of the
Dieudonn\'{e}-Grothendieck theorem which we state for the sake of
clearness.

\begin{lemma}\label{lemma-weakly-compact-on-singletons}
Let $K$ be a compact space, $X$ a Banach space and
$S:X^*\longrightarrow C(K)^*$ a weakly compact operator. Then,
for every bounded subset $B\subseteq X^*$, the set
$$
\big\{x\in K\,:\, \exists \phi\in B \text{ so that }
S(\phi)(\{x\})\neq0 \big\}
$$
is countable.
\end{lemma}

\begin{proof} Suppose that the set
$$
\big\{x\in K\,:\, \exists \phi\in B \text{ so that }
S(\phi)(\{x\})\neq0 \big\}
$$
is uncountable for some bounded set $B\subseteq X^*$. Then, there is
$\eps>0$ so that the set
$$
\big\{x\in K\,:\, \exists \phi\in B \text{ so that }
|S(\phi)(\{x\})|\geq\eps \big\}
$$
is infinite, which contradicts the fact of being $S(B)$
relatively weakly compact by the Dieudonn\'{e}-Grothendieck theorem
(see \cite[Theorem~VII.14]{Die}, for instance).
\end{proof}

\begin{lemma}\label{lemma-reduction-of-operators}
Let $K$ be a compact space, let $V_0$, $V_1$, and $V_2$ be open
nonempty subsets of $K$ compatible with $L$ such that
$\overline{V}_0\subseteq V_1$ and $\overline{V}_1\cap L=\emptyset$,
and let $T:C_E(K\|L)\longrightarrow C_E(K\|L)$ be a linear operator.
Then, there exists an operator $R: C(\overline{V}_1)\longrightarrow
C_E(\overline{V}_2\|L)$ such that
$$
R^*(\phi)|_{\overline{V}_0}=
(T^*P_{\overline{V}_2}^*)(\phi)|_{\overline{V}_0}
$$
for all $\phi\in C_E(\overline{V}_2\|L)^*$.
\end{lemma}

\begin{proof} Take a continuous function $h:K\longrightarrow [0,1]$ satisfying
$h|_{\overline{V}_0}\equiv 1$ and $h|_{(K\setminus{V}_1)}\equiv 0$,
and define the operator $R=P_{\overline{V}_2} T
I_{h,\overline{V}_1}$. Given $\phi \in C_E(\overline{V}_2\|L)^*$ and
$f\in C(\overline{V}_1)$ with $\supp(f)\subseteq \overline{V}_0$, by
parts (b) and (a) of Lemma~\ref{lemma-duals-inlusion-multiplication}
and using the facts $h|_{V_0}\equiv 1$ and
$\supp(f)\subseteq\overline{V}_0$, we can write
\begin{align*}
R^*(\phi)(f)&=I_{h,\overline{V}_1}^*
T^*P_{\overline{V}_2}^*(\phi)(f)= T^*
P_{\overline{V}_2}^*(\phi)(fh)=T^* P_{\overline{V}_2}^*(\phi)(f)
\end{align*}
which finishes the proof.
\end{proof}

We are ready to state and prove the main result of the section.

\begin{theorem}\label{thm-weakmultipliers-ex-non-complex}
Let $K$ be a perfect compact space such that all operators on $C(K)$
are weak multipliers, let $L\subseteq K$ be closed and nowhere
dense, and $E$ a closed subspace of $ C(L)$. Then, $C_E(K\|L)$ is
extremely non-complex.
\end{theorem}

\begin{proof} Fixed $T\in L\big(C_E(K\|L)\big)$,
we have to show that its square satisfies the Daugavet
equation. By Lemma~\ref{lemma-suf-condition}, it is enough to
prove that the set $\{x\in K\setminus L \ : \
g_{(T^2)^*}(x)\geq 0\}$ is dense in $K\setminus L$. To do so,
we proceed ad absurdum: suppose that there is an open set
$U_1\subseteq K$ such that $\overline{U}_1\cap L=\emptyset$ and
$g_{(T^2)^*}(x)<0$ for each $x\in U_1$. By going to a subset,
we may w.l.o.g. assume that $\overline{U}_1$ is a $G_\delta$
set. Therefore, we can find open sets $W_n\subseteq K$ such
that $\bigcap_{n\in\N} W_n=\overline{U}_1$,
$\overline{W}_{n+1}\subseteq W_n$, and $K\setminus W_n$ is the
closure of an open set containing $L$ for every $n\in \N$.
Next, we fix a nonempty open set $U_0\subseteq K$ with
$\overline{U}_0\subseteq U_1$, and we observe that it is
uncountable (since $K$ is perfect) so, there is $\eps>0$ such
that the set
$$
A=\{ x\in U_0\,:\, g_{(T^2)^*}(x)<-\eps\}
$$
is uncountable. Moreover, we claim that there is $n_0\in \N$ such
that the set
$$
B=\left\{ x\in A \,:\,
|T^*(\delta_x)|(W_{n_0}\setminus\overline{U}_1))<\frac{\eps}{2\|T\|}\right\}
$$
is uncountable. Indeed, fixed $x\in A$, the regularity of the
measure $T^*(\delta_x)$ implies that there is $n\in\N$ so that
$|T^*(\delta_x)|(W_{n}\setminus\overline{U}_1))<\frac{\eps}{2\|T\|}$
which gives us
$$
A=\bigcup_{n\in\N} \left\{ x\in A \,:\,
|T^*(\delta_x)|(W_{n}\setminus\overline{U}_1))<\frac{\eps}{2\|T\|}\right\}
$$
and the uncountability of $A$ finishes the argument.

For $x\in B$, we write
$$
\phi_x=T^*(\delta_x)|_{K\setminus(L\cup W_{n_0})}+T^*(\delta_x)_E\in
C_E\big((K\setminus W_{n_0})\|L\big)^*
$$
and we can decompose $T^*(\delta_x)$ as follows:
$$
T^*(\delta_x)=T^*(\delta_x)|_{\overline{U}_1}+
T^*(\delta_x)|_{W_{n_0}\setminus\overline{U}_1}+\phi_x\,.
$$
Hence, for every $x\in B$, we get
\begin{equation}\label{eq:CEKL-extremelynoncomplex1}
-\varepsilon>\bigl[(T^*)^2(\delta_x)\bigr](\{x\})=
T^*\bigl[T^*(\delta_x)|_{\overline{U}_1}+
T^*(\delta_x)|_{W_{n_0}\setminus\overline{U}_1}+\phi_x\bigr]\{x\}.
\end{equation}
However, the following claims show that this is impossible.

\noindent\emph{Claim~$1$}. $\bigl\|T^*\bigl[
 T^*(\delta_x)|_{W_{n_0}\setminus\overline{U}_1}\bigr]\bigr\|<\varepsilon/2$.

\noindent\emph{Proof of claim~$1$}. It follows obviously from the
choice of $n_0$.

\noindent\emph{Claim~$2$}. The function $x\longmapsto
T^*(T^*(\delta_x)|_{\overline{U}_1})(\{x\})$ is non-negative
for all but countably many $x\in B$.

\noindent\emph{Proof of claim~$2$}. By
Lemma~\ref{lemma-reduction-of-operators} applied to $V_0=U_1$,
$V_1=W_{n_0}$ and $V_2=U_1$, we obtain an operator $R:
C(\overline{W}_{n_0})\longrightarrow C(\overline{U}_1)$ such
that
$$
R^*(\psi)|_{\overline{U}_1}=T^*P_{\overline{U}_1}^*(\psi)|_{\overline{U}_1}
$$
for $\psi\in C(\overline{U}_1)^*$. On the other hand, by
Proposition~\ref{operators-on-open-of-assymetric} applied to
$R$ and $V_0=U_0$, $V_1=U_1$, $V_2=W_{n_0}$, we get a bounded
Borel function $g:\overline{U}_1\longrightarrow \R$ and a
weakly compact operator $S:C(\overline{U}_1)^*\longrightarrow
C(\overline{W}_{n_0})^*$ such that
$$
R^*(\mu)=g\mu+S(\mu)
$$
for every $\mu$ with support in $U_0$. In particular, for $x\in B$
we have
$$
T^*(\delta_x)|_{\overline{U}_1}
=T^*P_{\overline{U}_1}^*(\delta_x)|_{\overline{U}_1}=R^*(\delta_x)|_{\overline{U}_1}=
g\delta_x+
S(\delta_x)|_{\overline{U}_1}=g(x)\delta_x+S(\delta_x)|_{\overline{U}_1}
$$
and using this twice, we get
\begin{align*}
T^*(T^*(\delta_x)|_{\overline{U}_1})(\{x\}) &
=T^*\Big[g(x)\delta_x+S(\delta_x)|_{\overline{U}_1}\Big](\{x\}) =
g(x)T^*(\delta_x)(\{x\})+ T^*\big[ S(\delta_x)|_{\overline{U}_1}\big](\{x\})\\
&=g(x)T^*(\delta_x)|_{\overline{U}_1}(\{x\})+R^*\big[ S(\delta_x)|_{\overline{U}_1}\big](\{x\})\\
&=g(x)^2+g(x)S(\delta_x)|_{\overline{U}_1}(\{x\})+\big(gS(\delta_x)|_{\overline{U}_1}\big)(\{x\})
+S\big[ S(\delta_x)|_{\overline{U}_1}\big](\{x\})\\
&=g(x)^2+(2gP^{\overline{U}_1} S)(\delta_x)(\{x\})+(S
P^{\overline{U}_1} S)(\delta_x)(\{x\})
\end{align*}
for every $x\in B$. Finally, since $gP^{\overline{U}_1} S$ and $S
P^{\overline{U}_1} S$ are weakly compact operators, we conclude by
Lemma~\ref{lemma-weakly-compact-on-singletons} that
$(gP^{\overline{U}_1} S)(\delta_x)(\{x\})$ and $(S
P^{\overline{U}_1} S)(\delta_x)(\{x\})$ are zero for all but
countably many $x$, completing the proof of the claim.

\noindent\emph{Claim~$3$}. $T^*(\phi_x)(\{x\})=0$ for all but
countably many $x\in U_0$.

\noindent\emph{Proof of claim~$3$}. By
Lemma~\ref{lemma-reduction-of-operators} applied to $V_0=U_0$,
$V_1=U_1$, and $V_2=K\setminus \overline{W}_{n_0+1}$ we obtain an
operator $R: C(\overline{U}_1)\longrightarrow C_E\big((K\setminus
W_{n_0+1})\|L\big)$ such that $R^*(\phi)(\{x\})=T^*P_{K\setminus
W_{n_0}}^*(\phi)(\{x\})$ for $x\in U_0$ and $\phi\in
C_E\big((K\setminus W_{n_0+1})\|L\big)^*$. We denote
$J:C_E\big((K\setminus W_{n_0+1})\|L\big)\longrightarrow
C(K\setminus W_{n_0+1})$ the inclusion operator and we apply
Corollary~\ref{cor-operators-between-disjoint-sets} for the operator
$JR$ and the open sets $V_0=K\setminus\overline{W}_{n_0}$,
$V_1=K\setminus\overline{W}_{n_0+1}$, and $V_2=U_1$ to obtain that
the operator $P_{K\setminus W_{n_0}}JR$ is weakly compact.

Besides, we recall that $\phi_x\in C_E\big((K\setminus
W_{n_0}\|L)\big)^*$ and that it can be viewed as an element of
$C_E\big((K\setminus W_{n_0+1}\|L)\big)^*$ by just extending it by
zero outside $K\setminus W_{n_0}$. For $x\in U_0$ we take
$\widetilde{\phi}_x$ a Hahn-Banach extension of $\phi_x$ to
$C(K\setminus W_{n_0})$ and we observe that $J^*P_{K\setminus
W_{n_0}}^*(\widetilde{\phi}_x)=\phi_x$. Indeed, for $f\in
C_E\bigl((K\setminus W_{n_0+1})\|L\bigr)$ we have that
\begin{align*}
J^*P_{K\setminus
W_{n_0}}^*(\widetilde{\phi}_x)(f)&=\widetilde{\phi}_x\bigl(P_{K\setminus
W_{n_0}}(Jf)\bigr)=\widetilde{\phi}_x\bigl(P_{K\setminus
W_{n_0}}(f)\bigr)\\&=\widetilde{\phi}_x\bigl(f|_{K\setminus
W_{n_0}}\bigr)=\phi_x\bigl(f|_{K\setminus W_{n_0}}\bigr)=\phi_x(f).
\end{align*}
Therefore, for $x\in U_0$, we can write
\begin{align*}
T^*(\phi_x)(\{x\})&=T^*P_{K\setminus
W_{n_0}}^*(\phi_x)(\{x\})=R^*(\phi_x)(\{x\})\\&=R^*J^*P_{K\setminus
W_{n_0}}^*(\widetilde{\phi}_x)(\{x\})=(P_{K\setminus
W_{n_0}}JR)^*(\widetilde{\phi}_x)(\{x\})
\end{align*}
where we are identifying $\phi_x$ with its extension by zero to $K$.
Now the proof of the claim is finished by just applying
Lemma~\ref{lemma-weakly-compact-on-singletons} to the operator
$P_{K\setminus W_{n_0}}JR$.

Finally, the claims obviously contradict
\eqref{eq:CEKL-extremelynoncomplex1} completing the proof of the
theorem.
\end{proof}

When $E=\{0\}$, we get a sufficient condition to get that a space
of the form $C_0(K\setminus L)$ is extremely non-complex.

\begin{corollary}
Let $K$ be a compact space such that all operators on $C(K)$
are weak multipliers. Suppose $L\subseteq K$ is closed and
nowhere dense. Then, $C_0(K\setminus L)$ is extremely
non-complex.
\end{corollary}

To show that there are extremely non-complex spaces of the form
$C_E(K\|L)$ which are not isomorphic to the $C(K')$ spaces, we
need the following (well-known) result which allows us to
construct spaces $C_E(K\|L)$ for every perfect separable compact
space $K$ and every separable Banach space $E$.

\begin{lemma}\label{lemma:LcontinuouslymappedontoCantorset}
Let $K$ be a perfect compact space. Then:
\begin{enumerate}
\item[(a)] There is a nowhere dense closed subset $L\subset
    K$ such that $L$ can be continuously mapped onto the
    Cantor set.
\item[(b)] Therefore, every separable Banach space $E$ is
    (isometrically isomorphic to) a subspace of $C(L)$.
\end{enumerate}
\end{lemma}

\begin{proof}
(a). As $K$ is perfect, given an nonempty open subset $U$ in $K$
and $x\in U$, there are two nonempty open subsets $V_1, V_2$ of
$U$ such that
$$
\overline{V}_1\cap \overline{V}_2=\emptyset \qquad \text{and}
\qquad x\not\in \overline{V}_i \ (i=1,2).
$$
This allows us to construct a family of open sets $U_s$ for $s\in
\{0,1\}^{<\omega}$ such that
$$
U_{\emptyset}=K, \quad
\overline{U}_{s^\frown0} \cap \overline{U}_{s^\frown1}=\emptyset, \quad
\overline{U}_{s^\frown0}, \overline{U}_{s^\frown1} \subseteq
U_{s}, \quad \text{ and } \quad
U_{s}\setminus[\overline{U}_{s^\frown0}\cup \overline{U}_{s^\frown1}]\not=\emptyset.
$$
Take any point $y_{s}$ in the above difference. Define  $L$ to be
the set of all the accumulation points of the set  $\{y_s: s\in
\{0,1\}^{<\omega}\}$.

For $n\in \N$, let $f_n:K\longrightarrow [0,1]$ be such continuous
functions that for all $s\in \{0,1\}^n$ we have
$f_n[\overline{U}_{s^\frown0}]=\{0\}$ and
$f_n[\overline{U}_{s^\frown1}]=\{1\}$ which can be easily obtained
since $\overline{U}_{s}\cap \overline{U}_{s'}=\emptyset$ if
$s,s'\in \{0,1\}^n$ are distinct. Let $f:K\longrightarrow
[0,1]^{\N}$ be defined by $f(x)(n)=f_n(x)$. We claim that $f|_L$
satisfies the lemma. One can easily check that $f$ is continuous
since $f_n$s are continuous.

Because each $\overline{U}_s$ contains infinitely many points
$y_t$, we have that $L\cap \overline{U}_s\not=\emptyset$ for each
$s\in \{0,1\}^{<\omega}$. Note that if $x\in \overline{U}_s$ and
$s\in \{0,1\}^n$, then $f(x)|\{0,..., n-1\}=s$. So, as the image
of $L$ under $f$, is closed, it contains $\{0,1\}^{\N}$.

On the other hand if $x\in L$, then for each $n\in \N$ there is
$s\in \{0,1\}^{n}$ such that $x\in \overline{U}_s$, this is
because $K\setminus \bigcup\{\overline{U}_{s}\,:\,|s|=n\}$
contains only finitely many points $y_t$ and hence no element of
$L$. Thus $f_n(x)\in \{0,1\}$ if $x\in L$ which gives that
$f[L]\subseteq \{0,1\}^{\N}$, which together with the previous
observation gives that $f[L]=\{0,1\}^{\N}$.

Finally let us prove that  $L$ has empty interior, and so, as a
closed set, it is nowhere dense.  It is enough to see that $L$ has
empty interior in the subspace $\{y_s: s\in
\{0,1\}^{<\omega}\}\cup L$. This is true since $\{y_s: s\in
\{0,1\}^{<\omega}\}$ is dense and open in $\{y_s: s\in
\{0,1\}^{<\omega}\}\cup L$,
 as each
point $y_s$ is isolated by $U_{s}\setminus
[\overline{U}_{s^\frown0} \cup \overline{U}_{s^\frown1}]$ from the
remaining points.

(b). Since the function $f|_L:L\longrightarrow 2^\omega$ of the
above item is continuous and surjective, $C(2^\omega)$ embeds
isometrically into $C(L)$ by just composing every element in
$C(2^\omega)$ with $f|_L$. Since every separable Banach space $E$
embeds isometrically into $C(2^\omega)$ (Banach-Mazur theorem), we
get $E\subseteq C(L)$ isometrically.
\end{proof}

\begin{remark} If $K$ is a compact space such that all
operators on $C(K)$ are weak multipliers, then it is easier to prove
the existence of $L\subseteq K$ closed nowhere dense which maps onto
$[0,1]$ giving also (b) above. Indeed, $C(K)$ is a Grothendieck
space by \cite[Theorem~2.4]{Koszmider}, so $K$ has no convergent
sequence (otherwise it would give rise to a complemented copy of
$c_0$ contradicting the Grothendieck property). Now, take any
discrete sequence $\{x_n\,:\,n\in \N\}\subseteq K$ and consider the
set $L$ of all its accumulation points. Then, $L$ is perfect because
an isolated point would produce a convergent subsequence of
$\{x_n\,:\,n\in \N\}$, so $L$ continuously maps onto $[0,1]$
\cite[Theorem~8.5.4]{Sem}. To see that $L$ is nowhere dense we use
the discreteness of $\{x_n\,:\,n\in \N\}$. If $U\subseteq K$ is open
and intersects $L$, then there is $n_0\in \N$ such that $x_{n_0}\in
U$; but by the discreteness of $\{x_n\,:\,n\in N\}$, there is an
open neighborhood $V$ of $x_{n_0}$ not containing the remaining
$x_n$'s and hence, disjoint from $L$. Therefore, $V\cap U$ is an
open subset of $U$ disjoint with $L$, proving that $L$ is nowhere
dense.
\end{remark}

Now, we take a perfect compact space $K$ such that every
operator on $C(K)$ is a weak multiplier \cite{Koszmider}, and
we use Lemma~\ref{lemma:LcontinuouslymappedontoCantorset} to
find a nowhere dense closed subset $L$ such that $C(L)$
contains isometric copies of every separable Banach space.
Then, for every $E\subset C(L)$, $C_E(K\|L)$ is extremely
non-complex by Theorem~\ref{thm-weakmultipliers-ex-non-complex}
and $C_E(K\|L)^*=C_0(K\|L)^*\oplus_1 E^*$ by
Lemma~\ref{lemma-decomposition-dual}. If $E$ is
infinite-dimensional and reflexive, $C_E(K\|L)$ is not
isomorphic to a $C(K')$ space, since $C(K')^*$ never contains
complemented infinite-dimensional reflexive subspaces (see
\cite[Proposition~5.6.1]{Albiac-Kalton}, for instance). Let us
state all what we have proved.

\begin{example}$ $
\begin{enumerate}
\item[(a)] {\slshape For every separable Banach space $E$,
    there is an extremely non-complex Banach space
    $C_E(K\|L)$ such that $E^*$ is an $L$-summand in
    $C_E(K\|L)^*$.\ }
\item[(b)] If $E$ is infinite-dimensional and reflexive,
    then such $C_E(K\|L)$ is not isomorphic to any $C(K')$
    space.
\item[(c)] Therefore, {\slshape there are extremely
    non-complex Banach spaces which are not isomorphic to
    $C(K)$ spaces.\ }
\end{enumerate}
\end{example}

We finish the section commenting that some $C(K')$ spaces with
many operators which can be viewed as $C_E(K\|L)$ spaces where
$C(K)$ has few operators and for which our previous results
apply.

\begin{remark}\label{view-as-cofk} Let $L\subseteq K$ be a
nowhere dense subset of a compact $K$ as before. Consider the
topological quotient map $q: K\longrightarrow K_L$, where $K_L$ is
obtained from $K$ by identifying all points of $L$ to one point. The
canonical isometric embedding $I_q$ of $C(K_L)$ into $C(K)$ defined
by $I_q(f)=f\circ q$ has the image equal to the subspace of $C(K)$
consisting of all functions constant on $L$. Thus $C(K_L)$ is
isometric to $C_E(K\|L)$, where $E$ is the subspace of $C(L)$ of all
constant functions.  Hence, by the results above, if all operators
on $C(K)$ are weak multipliers, then all spaces of the form $C(K_L)$
are extremely non-complex. It turns out that the spaces of
\cite{KMM} can be realized as spaces of this form. In particular,
there are extremely non-complex spaces of the form $C_E(K\|L)$ which
have many operators besides weak multipliers. For example, take $K$
such that all operators on $C(K)$ are weak multipliers. Choose a
discrete sequence $(x_n)$ in $K$ and let $L$ be the set of its
accumulation points. Then the sequence $(x_n)$ has a unique
accumulation point in $K_L$, that is, it is a convergent sequence.
By a well known fact, this means that $C(K_L)\equiv C_E(K\|L)$ has a
complemented copy of $c_0$ and so is not Grothendieck, hence it has
more operators than weak multipliers by results of \cite{Koszmider}
(actually, many operators which are not weak multipliers can be
directly obtained from automorphisms of the complemented copy of
$c_0$ generated by permutations of the natural numbers).
\end{remark}

\section{Isometries on extremely non-complex spaces}\label{sec:isometries-on-X}

The following result shows that the group of isometries of an
extremely non complex Banach space is a discrete Boolean group.

\begin{theorem}\label{thm-isometries-form-boolean-group}
Let $X$ be an extremely non-complex Banach space. Then
\begin{enumerate}
\item[(a)] If $T\in \iso(X)$, then $T^2=\Id$.
\item[(b)] As a consequence, for every $T_1,T_2\in
    \iso(X)$, $T_1T_2=T_2T_1$.
\item[(c)] For every $T_1,T_2\in \iso(X)$, $\|T_1-T_2\|\in
    \{0,2\}$.
\end{enumerate}
\end{theorem}

\begin{proof}
(a). Given $T\in \iso(X)$, we define the operator
$S=\frac{1}{\sqrt{2}}\bigl(T-T^{-1}\bigr)$ and we observe that
$S^2= \frac{1}{2}\,T^2-\Id+\frac{1}{2}\,T^{-2}$. Since $X$ is
extremely non-complex, we get
$$
1+\|S^2\|=\|\Id+S^2\|=\left\|\frac{1}{2}\,T^2+\frac{1}{2}\,T^{-2}\right\|\leq
1
$$
and, therefore, $S^2=0$. This gives us that $\Id=\frac12
T^2+\frac12 T^{-2}$. Finally, since $\Id$ is an extreme point
of $L(X)$ (see \cite[Proposition~1.6.6]{sakai}, for instance)
and $\|T^2\|\leq 1$, $\|T^{-2}\|\leq 1$, we get $T^2=\Id$.

(b). Commutativity comes routinely from the first part since
$T_1T_2\in \iso(X)$, so
$$
\Id=\bigl(T_1 T_2\bigr)^2=T_1 T_2 T_1 T_2
$$
which finishes the proof by just multiplying by $T_1$ from the left
and by $T_2$ from the right.

(c). We start observing that $\|\Id-T\|\in\{0,2\}$ for every
$T\in \iso(X)$. Indeed, from (a) we have
$$
\bigl(\Id-T\bigr)^2=\Id+\Id-2T=2(\Id-T),
$$
which gives us that
$$
2\|\Id-T\|=\left\|(\Id - T)^2\right\|\leq\|\Id-T\|^2.
$$
Therefore, we get either $\|\Id-T\|=0$ or $\|\Id-T\|\geq2$.
Now, if $T_1,T_2\in \iso(X)$ we observe that
\begin{equation*}
\|T_1-T_2\|=\|T_1(\Id-T_1T_2)\|=\|\Id-T_1T_2\|\in\{0,2\}.\qedhere
\end{equation*}
\end{proof}

As an immediate consequence we obtain the following result. Let
us observe that there is no topological consideration on the
semigroup.

\begin{corollary}\label{cor-grups-of-isometries} If $X$ is an extremely
non-complex Banach space and $\Phi:\R^+_0\longrightarrow
\iso(X)$ is a one-parameter semigroup, then
$\Phi(\R^+_0)=\{\Id\}$.
\end{corollary}

\begin{proof}
Just observe that $\Phi(t)=\Phi(t/2+t/2)=\Phi(t/2)^2=\Id$ for
every $t\in \R^+_0$.
\end{proof}

Let $X$ be a Banach space. A projection $P\in L(X)$ is said to
be \emph{unconditional} if $2P-\Id\in \iso(X)$ (equivalently,
$\|2P-\Id\|=1$). We write $\mathrm{Unc}(X)$ for the set of
unconditional projections on $X$. It is straightforward to show
that $P\in \mathrm{Unc}(X)$ if and only if $P=\frac12(\Id+T)$
for some $T\in \iso(X)$ with $T^2=\Id$. It is then immediate
that $\mathrm{Unc}(X)$ identifies with $\{T\in \iso(X)\,:\,
T^2=\Id\}$ and both sets are Boolean groups: the group
operation in $\{T\in \iso(X)\,:\, T^2=\Id\}$ is just the
composition and so the group operation in $\mathrm{Unc}(X)$ is
$$
(P_1,P_2)\longmapsto P_1+P_2-P_1P_2.
$$
It also follows that all unconditional projections commute.

If $X$ is extremely non-complex, the set $\{T\in \iso(X)\,:\,
T^2=\Id\}$ is the whole $\iso(X)$
(Theorem~\ref{thm-isometries-form-boolean-group}). We summarize
all of this in the next result, where we will also discuss when
these Boolean groups are actually Boolean algebras. The proof
is completely straightforward. We refer the reader to the book
\cite[\S1.8]{Kop}  for background on Boolean algebras of projections.

\begin{prop}
Let $X$ be an extremely non-complex Banach space.
\begin{enumerate}
\item[(a)] $\iso(X)$ is a Boolean group for the composition
    operation.
\item[(b)] $\mathrm{Unc}(X)$ is (equivalently, $\iso(X)$ is
    isomorphic to) a Boolean algebra if, and only if,
    $P_1P_2\in \mathrm{Unc}(X)$ for every $P_1,P_2\in
    \mathrm{Unc}(X)$ if, and only if, $\|\Id + T_1 + T_2 -
    T_1T_2\|=2$ for every $T_1,T_2\in \iso(X)$.
\end{enumerate}
\end{prop}

We will show later that for many examples of extremely
non-complex Banach spaces the set of unconditional projections
is a Boolean algebra, but we do not know if this always
happens.

\section{Surjective isometries on extremely non-complex
$C_E(K\|L)$ spaces}\label{sec:isomtries-CEKL}

Our aim in this section is to describe the group of isometries
of the spaces $C_E(K\|L)$ when they are extremely non-complex.
We will deduce all the results from the following theorem.

\begin{theorem}\label{thm-form-of-isometries-in-CE(K)}
Suppose that the space $C_E(K\|L)$ is extremely non-complex.
Then, for every $T\in\iso\bigl(C_E(K\|L)\bigr)$ there is a
continuous function $\theta:K\setminus L\longrightarrow
\{-1,1\}$ such that
$$
[T(f)](x)=\theta(x) f(x)
$$
for all $x\in K\setminus L$ and $f\in C_E(K\|L)$.
\end{theorem}

\begin{proof} We divide the proof into several claims.

\noindent \emph{Claim~$1$}. The set $D_0=\{x\in K\setminus
L\,:\ \exists\, y\in K\setminus L,\, \theta_0\in\{-1,1\}\
\text{with} \ T^*(\delta_x)=\theta_0\delta_y\}$ is dense in
$K$.

\noindent\emph{Proof of claim~$1$}. Let $W$ be a nonempty open
subset of $K$. Since $K\setminus L$ is open and dense in $K$,
there is $V$ nonempty and open satisfying $V\subseteq W\cap
(K\setminus L)$. Now, $\{\delta_x\, : \, x\in V\}$ is a subset
of $\ext[C_E(K\|L)]$ by Lemma~\ref{lemma:1-norming-CEKL}, and
it is easy to check that it is weak$^*$ open in
$\ext[C_E(K\|L)^*]$ (indeed, for $x_0\in V$ take a non-negative
$f\in C_0(K\|L)$ such that $f(x_0)=1$ and $f(K\setminus V)=0$,
and observe that $\delta_{x_0}\in \{\delta_x\,:\,
\delta_x(f)>1/2\}\subset \{\delta_x\, : \, x\in V\}$). Now,
being $T^*$ a weak$^*$ continuous surjective isometry, the
mapping
$$
T^*:\left(\ext[C_E(K\|L)^*],
w^*\right)\longrightarrow\left(\ext[C_E(K\|L)^*], w^*\right)
$$
is a homeomorphism and so, the set $\{T^*(\delta_x)\, : \, x\in
V\}$ is weak$^*$ open in $\ext[C_E(K\|L)^*]$. Since, by
Lemma~\ref{lemma:1-norming-CEKL}, the set
$\bigl\{\theta\,\delta_y\, : \, y\in K\setminus
L,\,\theta\in\{-1,1\}\bigr\}$ is weak$^*$ dense in
$\ext[C_E(K\|L)^*]$, there are $x\in V$, $y\in K\setminus L$,
and $\theta_0\in\{-1,1\}$ such that $T^*(\delta_x)=\theta_0
\delta_y$, which implies $x\in V\cap D_0\subseteq W\cap D_0$,
finishing the proof of claim~$1$.

Now, we can consider functions $\phi:D_0\longrightarrow D_0$
and $\theta: D_0\longrightarrow \{-1,1\}$ such that
\begin{equation}\label{eq:thm-isometries-eps-phi}
T^*(\delta_x)=\theta(x)\,\delta_{\phi(x)}
\end{equation}
for all $x\in D_0$. Since $T^2=\Id$ by
Theorem~\ref{thm-isometries-form-boolean-group}, if $x\in D_0$
and $T^*(\delta_x)=\pm\delta_y$, then
$T^*(\delta_y)=\pm\delta_x$ and so $y\in D_0$. Therefore,
$\phi$ is a well defined function from $D_0$ into itself.
Moreover, it also follows that
\begin{equation}\label{eq:thm-isometries-eps-phi2}
\phi^2(x)=x \quad \text{and}\quad
\theta(x)\,\theta(\phi(x))=1 \qquad  (x\in D_0).
\end{equation}
Indeed, given $x\in D_0$, we use the fact that $(T^*)^2=\Id$ to
get
$$
\delta_x=T^*(T^*(\delta_x))=T^*(\theta(x)\delta_{\phi(x)})=
\theta(x)\theta(\phi(x))\,\delta_{\phi^2(x)}.
$$

\noindent \emph{Claim~$2$}. $\phi$ is a homeomorphism of $D_0$.

\noindent\emph{Proof of claim~$2$}. As $\phi^2$ is the identity
on $D_0$, it is enough to prove that $\phi$ is continuous. To
do so, fixed $x_0\in D_0$ and an open subset $W$ of $K\setminus
L$ with $\phi(x_0)\in W$, we have to show that $\phi^{-1}(W\cap
D_0)$ is a neighborhood of $x_0$ in $D_0$. Indeed, we consider
a continuous function $f_0\in C_0(K\|L) \subseteq C_E(K\|L)$
such that
$$
f_0(\phi(x_0))=1=\|f_0\| \qquad \text{and} \qquad f_0\equiv 0 \text{ in }K\setminus W.
$$
Since the mapping
$$
x\longmapsto [T^*(\delta_x)](f_0)=\theta(x)f_0(\phi(x)) \qquad \bigl(x\in D_0\bigr)
$$
is continuous at $x_0$, there is an open neighborhood $U_0$ of
$x_0$ such that
$$
\bigl||f_0(\phi(x))| - 1 \bigr|\leq \bigl|\theta(x)f_0(\phi(x)) - \theta(x_0)f_0(\phi(x_0)) \bigr|<\frac12 \quad
\bigl(x\in U_0\cap D_0\bigr).
$$
Since $f_0\equiv 0$ outside $W$, we get that
$
U_0\cap D_0\subseteq \phi^{-1}(W\cap D_0).
$

\noindent\emph{Claim~$3$}. $\phi(x)=x$ for all $x\in D_0$.

\noindent\emph{Proof of claim~$3$}. Suppose otherwise that
there are $x_0, y_0\in D_0$ such that $\phi(x_0)=y_0\neq x_0$.
Let $V_i\subseteq \overline{V}_i\subseteq K\setminus L$ with
$i=1,2$ be open subsets of $K$ satisfying
$$
x_0\in V_1, \quad y_0\in V_2, \quad
\overline{V}_1\cap\overline{V}_2=\emptyset, \quad \text{and}\quad
V_1\cap D_0\subseteq \phi^{-1}(V_2\cap D_0).
$$
The last condition obviously implies $\phi(V_1\cap
D_0)\subseteq V_2$ and, since $\phi$ is a homeomorphism of
$D_0$, it follows that $\phi(V_1\cap D_0)$ is open in $D_0$.
Therefore, we can find $V_0\subseteq V_2$ an open subset of $K$
such that $V_0\cap D_0=\phi(V_1\cap D_0)$. Then, we may find
$g\in C_0(K\|L)\subseteq C_E(K\|L)$ satisfying $g(x_0)=1$,
$g(y_0)=-1$, $g(x)\in [-1,0]$ for $x\in V_0$, $g(x)\in [0,1]$
for $x\in V_1$, and $g(x)=0$ for $x\not\in V_1\cup V_0$. In
particular, for $x\in D_0$, we have that
$$
g(x)\,g(\phi(x))\in [-1,0].
$$
Next, we define the operator $T_g:C_E(K\|L)\longrightarrow
C_E(K\|L)$ by $T_g(f)=gf$, which is well defined (since
$g|_L\equiv 0$) and satisfies $T_g^*(\delta_x)=g(x)\delta_x$
for each $x\in K\setminus L$. Finally, we consider the
composition $S=T_g T$ and, for $x\in D_0$, we use
\eqref{eq:thm-isometries-eps-phi} and
\eqref{eq:thm-isometries-eps-phi2} to write
$$
(S^*)^2(\delta_x)=T^*(T_g^*(T^*(T_g^*(\delta_x))))=
\theta(x)\theta(\phi(x))g(x)g(\phi(x))\delta_{\phi^2(x)}=
g(x)g(\phi(x))\delta_x.
$$
This, together with our choice of $g$, tells us that
$$
\bigl\|[\Id+(S^*)^2](\delta_x)\bigr\|\leq 1 \qquad \bigl(x\in D_0\bigr).
$$
As $D_0$ is dense in $K$ by claim~$1$ and $\Id+(S^*)^2$ is
weak$^*$-continuous, we deduce that $\|\Id+(S^*)^2\|\leq 1$.
Now, the fact that $C_E(K\|L)$ is extremely non-complex implies
that $S^2=0$ which is a contradiction since
$(S^*)^2(\delta_{x_0})=-\delta_{x_0}\not=0$.

\noindent\emph{Claim~$4$}. $D_0=K\setminus L$.

\noindent\emph{Proof of claim~$4$}. Let us fix $x_0\in
K\setminus L$. Since $D_0$ is dense in $K\setminus L$, we may
find a net $(x_\lambda)_{\lambda\in\Lambda}$ in $D_0$ such that
$(x_\lambda)_{\lambda\in\Lambda}\longrightarrow x_0$, so
$\bigl(T^*(\delta_{x_\lambda})\bigr)_{\lambda\in\Lambda}\longrightarrow
T^*(\delta_{x_0})$. But $T^*(\delta_{x_\lambda})
=\theta(x_\lambda)\,\delta_{x_\lambda}$, so the only possible
accumulation points of the net
$\bigl(T^*(\delta_{x_\lambda})\bigr)_{\lambda\in\Lambda}$ are
$+\delta_{x_0}$ and $-\delta_{x_0}$. Therefore, $x_0\in D_0$ as
claimed.


\noindent\emph{Claim~$5$}. $\theta$ is continuous on
$K\setminus L$.

\noindent\emph{Proof of claim~$5$}. We fix $x_0\in K\setminus
L$ and an open subset $W$ of $K$ such that $x_0\in W\subseteq
\overline{W}\subseteq K\setminus L$ and we take a function
$f\in C_0(K\|L)\subseteq C_E(K\|L)$ satisfying
$f|_{\overline{W}}\equiv 1$. Since the mapping
$$
x\longmapsto \psi(x)=[T^*(\delta_x)](f)=\theta(x)\,f(x) \qquad \bigl(x\in K\setminus
L\bigr)
$$
is continuous and $\psi|_{W}\equiv\theta|_{W}$, we get the
continuity of $\theta$ at $x_0$.
\end{proof}

We are now able to completely describe the set of surjective
isometries in some special cases. The first one covers the case
when $K$ and $K\setminus L$ are connected.

\begin{corollary}\label{cor:K-L-connected-isometries}
Let $K$ be a connected compact space such that $K\setminus L$ is
also connected. Suppose that $C_E(K\|L)$ is extremely non-complex.
Then, $\iso\bigl(C_E(K\|L)\bigr)=\{\Id,-\Id\}$.
\end{corollary}

\begin{proof}
Given $T\in \iso\bigl(C_E(K\|L)\bigr)$,
Theorem~\ref{thm-form-of-isometries-in-CE(K)} gives a continuous
function $\theta:K\setminus L\longrightarrow \{-1,1\}$ such that
$[T(f)](x)=\theta(x)\,f(x)$ for every $x\in K\setminus L$ and
every $f\in C_E(K\|L)$. If $K\setminus L$ is connected, there are
only two possible functions $\theta$. Being $L$ nowhere dense, the
values of $[T(f)](x)$ for $x\in K\setminus L$ determine completely
the function $T(f)$ for every $f\in C_E(K\|L)$. This gives only
two possible surjective isometries, $\Id$ and $-\Id$.
\end{proof}

\begin{corollary}\label{coro-M=emptyset}
Suppose $E$ is a subspace of $C(L)$ such that $C_E(K\|L)$ is
extremely non-complex and for every $x\in L$, there is $f\in E$
such that $f(x)\neq 0$. If $T\in \iso\bigl(C_E(K\|L)\bigr)$,
then there is a continuous function $\theta:
K\longrightarrow\{-1,1\}$ such that $T(f)=\theta\, f$ for all
$f\in C_E(K\|L)$.
\end{corollary}

\begin{proof}
By Theorem~\ref{thm-form-of-isometries-in-CE(K)}, we may find
$\theta': K\setminus L\longrightarrow \{-1,1\}$ continuous such
that
$$
[T(f)](x)=\theta'(x)\,f(x) \qquad \bigl(x\in K\setminus L,\ f\in C_E(K\|L)\bigr).
$$
First, we note that $\theta'$ can be extended to a continuous
function $\theta$ on $K$ (indeed, if $x\in L$, there is an open
neighborhood $U$ of $x$ on $K$ and an $f\in C_E(K\|L)$ such
that $f(y)\not= 0$ for every $y\in U$ and
$\frac{T(f)|_U}{f|_U}$ is a continuous function on $U$ which
extends $\theta'|_{U\setminus L}$). Now, for each $f\in
C_E(K\|L)$ we have
$$
[T(f)](x)=\theta(x)f(x) \qquad \bigl(x\in
K\setminus L\bigr),
$$
so $T(f)=\theta f$ since they are two continuous functions
which agree on a dense set.
\end{proof}

By just taking $E=C(L)$ in the above result, we get a
description of all surjective isometries on an extremely
non-complex $C(K)$ space. One direction is the above corollary,
the converse result is just a consequence of the classical
Banach-Stone theorem (see \cite[Theorem~2.1.1]{Fle-Jam1}, for
instance).

\begin{corollary}
Let $K$ be a perfect Hausdorff space such that $C(K)$ is extremely
non-complex. If $T\in \iso\bigl(C(K)\bigr)$, then there is a
continuous function $\theta:K\longrightarrow \{-1,1\}$ such that
$T(f)=\theta\,f$ for every $f\in C(K)$. Conversely, for every
continuous function $\theta':K\longrightarrow \{-1,1\}$, the
operator given by $T(f)=\theta'\,f$ for every $f\in C(K)$ is a
surjective isometry. In other words, $\iso\bigl(C(K)\bigr)$ is
isomorphic to the Boolean algebra of clopen subsets of $K$.
\end{corollary}

It follows from the above result and the Banach-Stone theorem
on the representation of surjective isometries on $C(K)$ (see
\cite[Theorem~1.2.2]{Fle-Jam1} for instance) that the only
homeomorphism of $K$ is the identity.

\begin{corollary}
Let $K$ be a perfect Hausdorff space such that $C(K)$ is
extremely non-complex. Then, the unique homeomorphism from $K$
onto $K$ is the identity.
\end{corollary}

We finish the section with the study of the opposite extreme case,
i.e.\ when $E=\{0\}$. Then, the hypothesis of
Corollary~\ref{coro-M=emptyset} are not satisfied, but we obtain a
description of the surjective isometries of the spaces
$C_0(K\|L)\equiv C_0(K\setminus L)$ directly from
Theorem~\ref{thm-form-of-isometries-in-CE(K)}. Again, the converse
result comes from the Banach-Stone theorem (see
\cite[Corollary~2.3.12]{Fle-Jam1} for instance).

\begin{corollary}
Let $K$ be a compact Hausdorff space, $L\subset K$ closed
nowhere dense, and suppose that $C_0(K\setminus L)$ is
extremely non-complex. If $T\in \iso\bigl(C_0(K\setminus
L)\bigr)$, then there is a continuous function
$\theta:K\setminus L\longrightarrow \{-1,1\}$ such that
$T(f)=\theta\,f$ for every $f\in C_0(K\setminus L)$.
Conversely, for every continuous function $\theta':K\setminus L
\longrightarrow \{-1,1\}$, the operator
$$
\bigl[T(f)\bigr](x)=\theta'(x)\,f(x) \qquad
\bigl(x\in K\setminus L,\ f\in C_0(K\setminus L)\bigr)
$$
is a surjective isometry. In other words,
$\iso\bigl(C_0(K\setminus L)\bigr)$ is isomorphic to the
Boolean algebra of clopen subsets of $K\setminus L$.
\end{corollary}

\begin{proof}
The first part is a direct consequence of
Theorem~\ref{thm-form-of-isometries-in-CE(K)} for $E=\{0\}$. For
the converse result, just observe that given any extension of
$\theta'$ to $L$, the product $\theta'\,f:K\longrightarrow \R$
does not depend on the extension, belongs to $C_0(K\|L)$ and
$\|\theta'\,f\|_\infty=\|f\|_\infty$.
\end{proof}

\section{The construction of the main example}\label{sec:main-example}
Our goal here is to construct a compact space $K$ and a nowhere
dense subset $L\subseteq K$ with very special properties which
will allow us to provide the main example on surjective
isometries and duality.

\begin{theorem}\label{th:K-L-connected-construction}
There exist a compact space $K$ and a closed nowhere dense
subset $L\subseteq K$ with the following properties:
\begin{enumerate}
\item[(a)] $K$ and $K\setminus L$ are connected.
\item[(b)] There is a continuous mapping $\phi$ from $L$
    onto the Cantor set.
\item[(c)] Every operator on $C(K)$ is a weak multiplier.
\end{enumerate}
\end{theorem}

\begin{proof}
$K$ is the compact space constructed in \cite[\S5]{Koszmider}.
The fact that all operators on $C(K)$ are weak multipliers is
given in \cite[Lemma~5.2]{Koszmider}.

We just need to find the appropriate $L$. We will assume the
familiarity of the reader with the above construction of $K$ in
\cite[\S~5]{Koszmider}. In particular, that  $K\subseteq
[0,1]^{2^\omega}$ is the inverse limit of $K_\alpha\subseteq
[0,1]^\alpha$ for $\alpha\leq 2^\omega$ where $K_1=[0,1]^2$.
For $\beta\leq\alpha\leq 2^\omega$ the projection from
$[0,1]^\alpha$ onto $[0,1]^\beta$ is denoted
$\pi_{\beta,\alpha}$.

Choose any $N\subseteq [0,1]^2$ which is a copy of a Cantor set
included in some subinterval of $[0,1]^2$. In particular, it is
compact nowhere dense perfect and such that $[0,1]^2\setminus
N$ is connected. Let $N_\alpha=\pi_{1,\alpha}[N]$. We claim
that $L=N_{2^{\omega}}$ works i.e., is nowhere dense in $K$ and
$K\setminus L$ is connected and there is a continuous mapping
of $L$ onto the Cantor set. The last part is clear as
$\pi_{1,2^\omega}$ sends $L$ onto $N$ which is a homeomorphic
copy of the Cantor set.

One proves by induction on $\alpha\leq 2^{\omega}$ that
$N_{\alpha}$ is nowhere dense in $K_\alpha$ and
$K_\alpha\setminus N_\alpha$ is connected. This is essentially
a generalization of \cite[Lemma~4.6]{Koszmider} from a finite
set to a nowhere dense set with a connected complement in
$[0,1]^2$.

\cite[Lemma~4.3.a]{Koszmider} says that being nowhere dense is
preserved when we pass by preimages from $K_\alpha$ to
$K_{\alpha+1}$ so, as the limit stage is trivial, it follows
that every $N_\alpha$ is nowhere dense in $K_\alpha$.
Therefore, $L$ is nowhere dense in $K$.

So we are left with showing that $K_\alpha\setminus N_\alpha$
are connected. As in \cite[Lemma~4.6]{Koszmider}, we prove by
induction on $\alpha$ that there are $M^n_\alpha\subseteq
K_\alpha$ such that
\begin{enumerate}
\item[1)]
    $\pi_{\alpha',\alpha}[M^n_{\alpha}]=M^n_{\alpha'}$ for
    $\alpha'\leq\alpha\leq\beta$,
\item[2)] $M^n_\alpha$'s are compact and connected,
\item[3)] $M^n_\alpha\cap N_\alpha=\emptyset$,
    $M^n_\alpha\subseteq M^{n+1}_\alpha$,
\item[4)] $\bigcup_{n\in N} M^n_\alpha$ is dense in
    $K_\alpha\setminus N_\alpha$.
\end{enumerate}
We start by choosing $M^n_1$ to satisfy 2) - 4) and such that
$[0,1]^2\setminus\bigcup_{n\in N} M^n_\alpha$ is $N=N_1\subseteq
[0,1]^2$. The rest of the argument is exactly as in the last part of
the proof of \cite[Lemma~4.6]{Koszmider}.
\end{proof}

We are now ready to present the main application of the results
of the paper.

\begin{theorem}\label{thm:everyE-CEKL-isometries}
For every separable Banach space $E$, there is a Banach space
$\widetilde{X}(E)$ such that
$\iso\bigl(\widetilde{X}(E)\bigr)=\{\Id,-\Id\}$ and
$\widetilde{X}(E)^*=E^*\oplus_1 Z$ for a suitable space $Z$. In
particular, $\iso\bigl(\widetilde{X}(E)^*\bigr)$ contains
$\iso(E^*)$ as a subgroup.
\end{theorem}

\begin{proof}
Consider the compact space $K$ and the nowhere dense closed subset
$L\subset K$ given in Theorem~\ref{th:K-L-connected-construction}.
As there is a surjective continuous function from $L$ to the Cantor
set, every separable Banach space $E$ is a subset of $C(L)$. Let
$\widetilde{X}(E)$ be $C_E(K\|L)$. Then, $\widetilde{X}(E)$ is
extremely non complex since every operator on $C(K)$ is a weak
multiplier and we may use
Theorem~\ref{thm-weakmultipliers-ex-non-complex}. Now, since
$K\setminus L$ is connected, we may apply
Corollary~\ref{cor:K-L-connected-isometries} to get that
$\iso\bigl(\widetilde{X}(E)\bigr)=\{\Id,-\Id\}$. Finally,
Lemma~\ref{lemma-decomposition-dual} gives us that
$\widetilde{X}(E)^*=E^*\oplus_1 C_0(K\|L)^*$ and so
$\iso\bigl(\widetilde{X}(E)^*\bigr)$ contains $\iso(E^*)$ as a
subgroup (see \cite[Proposition~2.4]{MarIso} for instance).
\end{proof}

Let us comment that all the spaces $\widetilde{X}(E)$ constructed
above are non-separable. We do not know whether separable spaces
with the same properties can be constructed.

The case $E=\ell_2$ in Theorem~\ref{thm:everyE-CEKL-isometries}
gives the following specially interesting example.

\begin{example}
{\slshape There is a Banach space $\widetilde{X}(\ell_2)$ such
that $\iso\bigl(\widetilde{X}(\ell_2)\bigr)=\{\Id,-\Id\}$ but
$\iso\bigl(\widetilde{X}(\ell_2)^*\bigr)$ contains
$\iso(\ell_2)$ as a subgroup. Therefore,
$\iso\bigl(\widetilde{X}(\ell_2)\bigr)$ is trivial, while
$\iso\bigl(\widetilde{X}(\ell_2)^*\bigr)$ contains infinitely
many uniformly continuous one-parameter semigroups of
surjective isometries.\ }
\end{example}

Recently, the second author of this paper constructed a Banach
space $X(\ell_2)$ such that $\iso\bigl(X(\ell_2)\bigr)$ does not
contain any \emph{uniformly continuous} one-parameter semigroup of
surjective isometries, while $\iso\bigl(X(\ell_2)^*\bigr)$
contains infinitely many of them \cite[Example~4.1]{MarIso}. But
it is not difficult to show that $\iso\bigl(X(\ell_2)\bigr)$ does
not reduce to $\{\Id,-\Id\}$ and, actually, it contains infinitely
many \emph{strongly continuous} one-parameter semigroups of
surjective isometries. We refer the reader to the books
\cite{Engel-Nagel,Engel-Nagel-shot} for background on
one-parameter semigroups of operators and to the monographs
\cite{Fle-Jam1,Fle-Jam2} for more information on isometries on
Banach spaces.

\vspace{5mm}

\noindent\textbf{Acknowledgements.} The authors would like to
thank Vladimir Kadets, Rafael Ortega, Rafael Pay\'{a} and Armando
Villena for fruitful conversations on the subject of the paper.
Part of this work was done while Miguel Mart\'{\i}n was visiting the
Technical University of \L odz in July 2008 and also while Piotr
Koszmider was visiting the University of Granada in September
2008. Both authors would like to thank the hosting institutions
for partial financial support.

\end{document}